\documentclass{amsart}
\usepackage[utf8]{inputenc}

\usepackage{amsfonts}
\usepackage{amsmath}
\usepackage{amssymb}
\usepackage{amsthm}
\usepackage{mathtools}

\newtheorem{theorem}{Theorem}
\newtheorem{lemma}[theorem]{Lemma}
\newtheorem{corollary}[theorem]{Corollary}

\DeclarePairedDelimiter\floor{\lfloor}{\rfloor}

\title[Solubility of Additive Quartic Forms]{Solubility of Additive Quartic Forms over Ramified Quadratic Extensions of $\mathbb{Q}_2$}
\date{\today}

\begin{document}

\author{Drew Duncan}
\address{Department of Mathematics, Computer Science, and Data Science\\
John Carroll University\\
University Heights, OH 44118}
\email{rduncan@jcu.edu}

\author{David B. Leep}
\address{Department of Mathematics\\
University of Kentucky\\
Lexington, KY 40506}
\email{leep@uky.edu}


\begin{abstract}
We determine the minimal number of variables $\Gamma^*(d, K)$ which guarantees a nontrivial solution for every additive form of degree $d=4$ over the four ramified quadratic extensions $\mathbb{Q}_2(\sqrt{2}), \mathbb{Q}_2(\sqrt{10}), \mathbb{Q}_2(\sqrt{-2}), \mathbb{Q}_2(\sqrt{-10}) $ of $\mathbb{Q}_2$.  In all four fields, we prove that $\Gamma^*(4,K) = 11$.  This is the first example of such a computation for a proper extension of $\mathbb{Q}_p$ where the degree is a power of $p$ greater than $p$.
\end{abstract}



\maketitle

\section{Introduction}

Homogeneous forms of the type
\begin{equation}
\label{form}
a_1x_1^d + a_2x_2^d + \ldots + a_s x_s^d
\end{equation} where $a_1, \ldots, a_s$ belong to some field $K$ are known as \textit{additive forms} of degree $d$.  It is clear that every such form has a \textit{trivial} zero, where each variable takes the value zero.  A famous conjecture of Artin claimed that any homogeneous form over a p-adic field $K$ (a finite extension of $\mathbb{Q}_p$, the field of p-adic numbers) of degree $d$ in $d^2+1$ variables has a \textit{nontrivial} zero regardless of the choice of coefficients from $K$.  Although many well-known counterexamples to this conjecture have been discovered (see \cite{MR197450}, \cite{greenberg1969lectures}, \cite{bartholdi_thesis}), none of them have been additive forms.  It has therefore been proposed that the conjecture holds when restricted to additive forms.  We will refer to this as Artin's Additive Form Conjecture.

Let $\Gamma^*(d, K)$ represent the minimum number of variables $s$ such that every form (\ref{form}) is guaranteed to have a nontrivial zero, regardless of the choice of $a_i \in K$.  In this language, Artin's Additive Form Conjecture posits that $\Gamma^*(d, K) \le d^2+1$. Davenport and Lewis introduced the method of \textit{contraction} in \cite{davenport1963homogeneous}, and used it to establish the truth of Artin's Additive Form Conjecture for every field of $p$-adic numbers $\mathbb{Q}_p$. Indeed, this bound is exact when $d=p-1$; i.e., $\Gamma^*(p-1, \mathbb{Q}_p) = (p-1)^2 + 1$.  They noted, however, that when $d \ne p-1$ very often a much lower value of $\Gamma^*(d, \mathbb{Q}_p)$ can be found.  Exact values of $\Gamma^*(d, \mathbb{Q}_p)$ are known only in a small number of cases (e.g., \cite{knapp_exact_gamma__MR2811557}, \cite{knapp_2adic_diagonal__MR3846799}), and even less is known in general about $\Gamma^*$ for algebraic extensions of $\mathbb{Q}_p$.  For unramified extensions of $K / \mathbb{Q}_p$, $p>2$, the bound $\Gamma^*(d, K) \le d^2 + 1$ was established in \cite{leep2018diagonal}, and it was recently shown in \cite{MGK} that this bound also holds for all degrees not a power of 2 over all seven quadratic extension of $\mathbb{Q}_2$.  We established exact values of $\Gamma^*(d, K)$ for the six ramified quadratic extensions of $\mathbb{Q}_2$ and any degree $d=2m$, $m$ odd (\cite{2020arXiv200509770D}, \cite{2020arXiv201006833D}).  Up to now, no results have been established for Artin's Additive Form Conjecture when the degree is a power of 2 at least 4 over proper extensions of $\mathbb{Q}_p$.  In this paper, we find exact values of $\Gamma^*(4,K)$ over four ramified quadratic extensions of $\mathbb{Q}_2$, demonstrating the following result.

\begin{theorem}
If $K \in \{\mathbb{Q}_2(\sqrt{2}), \mathbb{Q}_2(\sqrt{10}), \mathbb{Q}_2(\sqrt{-2}), \mathbb{Q}_2(\sqrt{-10})\}$, then  $$\Gamma^*(4, K) = 11.$$
\end{theorem}

Further, we conjecture that the following holds for the remaining two quadratic ramified extensions.

\bigskip

\textbf{Conjecture.}
\textit{Artin's Conjecture holds for all additive quartic forms over ramified quadratic extensions of $\mathbb{Q}_2$. In particular:
}
$$\Gamma^*(4, \mathbb{Q}_2(\sqrt{-1})) = 11,$$

$$\Gamma^*(4, \mathbb{Q}_2(\sqrt{-5})) = 9.$$

\bigskip

We have $\Gamma^*(4, \mathbb{Q}_2(\sqrt{-1})) \ge 11$ because
$$x_0^4 + x_1^4 + x_2^4 + x_3^4 + (\pi + \pi^3)(x_4^4 + x_5^4 + x_6^4 + x_7^4) + (\pi^2 + \pi^3)x_8^4 + (\pi^3 + \pi^4 + \pi^5)x_9^4$$
is anisotropic, and  $\Gamma^*(4, \mathbb{Q}_2(\sqrt{-5})) \ge 9$ because
$$x_0^4 + x_1^4 + x_2^4 + x_3^4 + \pi(x_4^4 + x_5^4 + x_6^4 + x_7^4)$$ is anisotropic, where $\pi$ is as in Table \ref{fields}.  Direct computation shows that neither of these forms has a primitive zero modulo $\pi^7$.

\section{Preliminaries}

\begin{table}[]
    \caption{Uniformizer and representation of 2}
    \centering
    \begin{tabular}{|c|c|c|}
        \hline 
         $K$ & $\pi$ & $2 \pmod{\pi^7}$ \rule{0pt}{2.6ex} \\
         \hline
         \rule{0pt}{2.6ex} \rule[-0.9ex]{0pt}{0pt}
         $\mathbb{Q}_2(\sqrt{2})$ & $\sqrt{2}$ & $\pi^2$ \rule{0pt}{2.6ex} \rule[-0.9ex]{0pt}{0pt}\\
         \hline
         $\mathbb{Q}_2(\sqrt{-2})$ & $\sqrt{-2}$ & $\pi^2 + \pi^4$ \rule{0pt}{2.6ex} \\
         \hline
         $\mathbb{Q}_2(\sqrt{10})$ & $\sqrt{10}$ & $\pi^2 + \pi^6$ \rule{0pt}{2.6ex} \\
         \hline
         $\mathbb{Q}_2(\sqrt{-10})$ & $\sqrt{-10}$ & $\pi^2 + \pi^4 + \pi^6$ \rule{0pt}{2.6ex} \\
         \hline
         $\mathbb{Q}_2(\sqrt{-1})$ & $1 + \sqrt{-1}$ & $\pi^2 + \pi^3 + \pi^5 + \pi^6$ \rule{0pt}{2.6ex} \\
         \hline
         $\mathbb{Q}_2(\sqrt{-5})$ & $1 + \sqrt{-5}$ & $\pi^2 + \pi^3 + \pi^4 + \pi^6$ \rule{0pt}{2.6ex} \\
         \hline
    \end{tabular}
    \label{fields}
\end{table}

Let $K$ denote one of the six ramified quadratic extensions of $\mathbb{Q}_2$, $\mathcal{O}$ denote its ring of integers, and $e=2$ denote its degree of ramification.  Without loss of generality, assume $a_i \in \mathcal{O}\backslash\{0\}$.  Let $\pi$ be a uniformizer (generator of the unique maximal ideal) of $\mathcal{O}$, so that any $c \in \mathcal{O}$ can be written $c = c_0 + c_1\pi + c_2\pi^2 + c_3\pi^3 + \ldots$, with $c_i \in \{0,1\}$.  The choices of uniformizer and corresponding representation of 2 used in this paper are listed in Table \ref{fields} (cf. this data modulo $\pi^5$ in Table 1 of \cite{knapp2016solubility}).

It is straightforward to see that the choices of $\pi$ for $\mathbb{Q}_2(\sqrt{2})$, $\mathbb{Q}_2(\sqrt{-2})$, 
$\mathbb{Q}_2(\sqrt{10})$, and  
$\mathbb{Q}_2(\sqrt{-10})$ are suitable; in each case $\pi^2$ gives 2 times an odd integer, and every odd integer is a unit in $\mathcal{O}$.  In the case of $\mathbb{Q}_2(\sqrt{2})$, $\pi^2 = 2$, and in $\mathbb{Q}_2(\sqrt{-2})$, $\pi^4 + \pi^2 = 2$ are exact representations of 2.  However, in the other four fields, the exact representations of 2 with respect to the choice of uniformizer have an infinite number of nonzero terms.  Note that $\pi^7 \mid 8\pi \mid 16$.  For $\pi = \sqrt{10}$, we have $16 | \pi^2 + \pi^6 - 2 = 1008$, and for $\pi = \sqrt{-10}$ we have $16 | \pi^2 + \pi^4 + \pi^6 - 2 = -912$.  For $\pi = 1 + \sqrt{-5}$ we have $\pi^2 + \pi^3 + \pi^4 + \pi^6 - 2= 152 + 40\sqrt{-5} = (8\pi)(\frac{22 - 7\sqrt{-5}}{3})$.  (Note that in this case, $\frac{1}{3} \in \mathcal{O}$.)  For $\pi = 1 + \sqrt{-1}$, we again have a closed form for the representation $2 = \pi^2 + \pi^3 + \frac{\pi^5}{1-\pi}$.

For the remainder of this chapter, let $K$ denote one of the ramified quadratic extensions $\mathbb{Q}_2(\sqrt{2}), \mathbb{Q}_2(\sqrt{10}), \mathbb{Q}_2(\sqrt{-2}), \mathbb{Q}_2(\sqrt{-10})$, assume that the number of variables $s$ in the form (\ref{form}) is 11, and the exponent $d$ is 4, i.e,
\begin{equation}
\label{quartic_form}
a_1x_1^{4} + a_2x_2^{4} + \ldots + a_{11}x_{11}^{4}.
\end{equation}

Factoring out the highest power of $\pi$, the coefficient $c$ of any variable $x$ can be written in the form $c = \pi^r(c_0 + c_1\pi + c_2\pi^2 + c_3\pi^3 + \ldots)$, $c_0 \neq 0$.  (Indeed, for the fields considered here, $c_0=1$.)  Such a variable is said to be at \textit{level $r$}, and we will refer to the value of $c_1$ as its \textit{$\pi$-coefficient}, and more generally $c_k$ as its \textit{$\pi^k$-coefficient}.  We will also have occasion to refer to selections of coefficients as a \textit{coefficient class}.

A set of constraints on a form will be referred to as a type, and any form satisfying those constraints will be said to be of that type.  In particular for this paper, the notation $(s_0, s_1, \ldots, s_\ell)$ will describe the type of a form which has \textit{at least} $s_i$ variables in level $i$ for each level $i$ listed.  Levels which are omitted in this notation may be assumed to contain as few as no variables.  For example, $(4, 2, 0, 1)$ indicates a form with at least four variables at level 0, at least two variables at level 1, as few as no variables at level 2, at least one variable in level 3, and as few as no variables at any higher level.

The change of variables $\pi^r x^d = \pi^{r-id}(\pi^i x)^d = \pi^{r-id}y^d$ for $i \in \mathbb{Z}$ replaces any variable with one at a level that differs from the level of the original variable by a multiple of $d$.  Intuitively, we can think of this as moving a variable up or down a multiple of $d$ levels.  Because this change of variables doesn't change whether or not the form has a nontrivial zero, we will often simply consider the level of a variable modulo $d$. Thus for $d=4$, it suffices to consider only forms which have all of their variables in levels 0 through 3, and so throughout this paper we will consider types of forms with only four levels listed, e.g., forms of type $(5, 0, 1, 1)$.

Multiplying a form by $\pi$ increases the level of each variable by one, and does not affect the existence of a nontrivial zero.  Considering the levels of variables modulo $d$, applying this transformation any number of times effects a cyclic permutation of the levels.  Therefore, we extend the type notation of the previous paragraph by considering all cyclic permutations of the type notation to denote the same type.  For example, a form of type $(1,1,5,0)$ will also be said to be of type $(5,0,1,1)$.  This transformation may also be used to arrange the variables in an order which is more convenient, as specified in Lemma \ref{normalization}.  We will refer to this process as \textit{normalization}. See Lemma 3 of \cite{davenport1963homogeneous} for a proof of the following Lemma.

\begin{lemma}
\label{normalization}
Given an additive form of degree $d$ in an arbitrary p-adic field $K$, let $s$ be the total number of variables and $s_i$ be the number of variables in level $i \pmod{d}$. By a change of variables, the form may be transformed to one with:
\begin{align}
\begin{split}
s_0 \ge \frac{s}{d},  
\end{split}
\begin{split}
s_0 + s_1 \ge \frac{2s}{d},
\end{split}
\begin{split}
\ldots,
\end{split}
\begin{split}
s_0 + \ldots + s_{d-1} = s
\end{split}
\end{align}
\end{lemma}

Consider two variables $x_1,x_2$ in the same level, and without loss of generality assume their coefficients $c_1,c_2$ are not divisible by $\pi$ (by the above cyclic permutation of levels).  Suppose there is an assignment $x_1 = b_1, x_2 = b_2$ with at least one of $b_1,b_2$ in $\mathcal{O}^\times$, such that $\pi^k | (a_1 b_1^d + a_2 b_2^d)$.  Then the change of variables $x_1=b_1 y, x_2 = b_2 y$ yields a form in which the two variables $x_1, x_2$ are replaced with the new variable $y$ at least $k$ levels higher.  If this new form has a nontrivial zero, then there is a nontrivial zero of the original form.  This transformation is known as a \textit{contraction}, and it is the key technique used to obtain all of the results that follow.

Contractions can be used to demonstrate the existence of a nontrivial zero by applying the following version of Hensel's Lemma, specialized for diagonal forms (see \cite{leep2018diagonal}, Theorem 2.3 for a proof):

\begin{lemma}[Hensel's Lemma]
\label{diagonal_hensels_lemma}
Let $\gamma = \begin{cases} 1 & \text{if $\tau = 0$} \\ \floor{\frac{e}{p-1}} + e\tau + 1 & \text{if $\tau \ge 1.$}
\end{cases}$\\
Let $d = mp^\tau$, where $(m, p) = 1$.  Suppose that $b,c \in \mathcal{O}^\times$ and that the congruence $cx^d \equiv b \bmod \pi^{\nu}$ has a solution $a \in \mathcal{O}$ for some $\nu \ge \gamma$.  Then, the congruence $cs^d \equiv b \bmod \pi^{\nu+1}$ has a solution $t$ where $t \equiv a \bmod \pi^{\nu - e\tau}$.  Consequently, the equation $cx^d = b$ has a solution in $\mathcal{O}$.
\end{lemma}

The aim is then to show that a series of contractions can be performed which ``raises" a variable $\gamma$ levels.  To simplify showing that such a series of contractions can be formed, one may first designate a level $k$.  Then, by showing that contractions involving at least one variable from the designated level produce a variable at level $k+\gamma$, a nontrivial solution follows.  Specializing Lemma \ref{diagonal_hensels_lemma} to the case $e = 2$, $\tau = 2$, $p = 2$, we obtain:

\begin{lemma}
\label{quartic_hensels_lemma}
Let $x_i$ be a variable of (\ref{quartic_form}) at level $k$.  Suppose that $x_i$ can be used in a contraction of variables (or one in a series of contractions) which produces a new variable at level at least $k+7$.  Then (\ref{quartic_form}) has a nontrivial zero.
\end{lemma}

We will henceforth refer to Lemma \ref{quartic_hensels_lemma} simply as \textit{Hensel's Lemma}.

In order to see how contractions behave in the four chosen fields, we compute the fourth powers modulo $\pi^7$.

\begin{lemma}
\label{powers}
The unit fourth powers modulo $\pi^7$ are

$\begin{cases}
1, 1+\pi^5 & \text{if } K \in \{\mathbb{Q}_2(\sqrt{2}), \mathbb{Q}_2(\sqrt{10})\} \\
1, 1+\pi^5 + \pi^6 & \text{if } K \in \{\mathbb{Q}_2(\sqrt{-2}), \mathbb{Q}_2(\sqrt{-10})\}.
\end{cases}$

\end{lemma}
\begin{proof}
Let $2 = \pi^2u$, where $\pi$ is as in Table \ref{fields} and $u \in \mathcal{O}^\times$, $a = a_0 + a_1\pi + a_2\pi^2 + \ldots \in \mathcal{O}$ with $a_i \in \{0,1\}$.  Note that for the four fields considered, $2 \equiv \pi^2 \pmod{\pi^4}$, i.e., $u \equiv 1 \pmod{\pi^2}$.  Then

\begin{align*}
    (1 + a\pi)^4 & = 
    1 + 4a\pi + 6a^2\pi^2 + 4a^3\pi^3 + a^4\pi^4\\
    &=1 + u^2a\pi^5 + (ua^2\pi^4 + u^2a^2\pi^6) + u^2a^3\pi^7 + a^4\pi^4 \\
    & \equiv 1 + (ua^2 + a^4)\pi^4 + u^2a\pi^5 + u^2a^2\pi^6 \\
    & \equiv 1 + (ua^2 + a^4)\pi^4 + a\pi^5 + a\pi^6 \\
    & \equiv 1 + (ua_0 + ua_1\pi^2 + a_0)\pi^4 + (a_0 + a_1\pi)\pi^5 + a_0\pi^6 \\
    & \equiv 1 + (u+1)a_0\pi^4 + a_0\pi^5 + a_0\pi^6
    \pmod{\pi^7}.
\end{align*}

If $K \in \{\mathbb{Q}_2(\sqrt{2}), \mathbb{Q}_2(\sqrt{10})\}$, then $u \equiv 1 \pmod{\pi^4}$, and so $$(1 + a\pi)^4 \equiv 1 + 2a_0\pi^4 + a_0\pi^5 + a_0 \pi^6 \equiv 1+a_0\pi^5\pmod{\pi^7}.$$ If $K \in \{\mathbb{Q}_2(\sqrt{-2}), \mathbb{Q}_2(\sqrt{-10})\}$, then $u \equiv 1 + \pi^2 \pmod{\pi^4}$, and so $$(1 + a\pi)^4 \equiv 1 + (2 +\pi^2)a_0\pi^4 + a_0\pi^5 + a_0 \pi^6 \equiv 1+a_0\pi^5+a_0\pi^6\pmod{\pi^7}.$$
\end{proof}

Consequently, there exists an $\alpha \in \mathcal{O}$ such that $\alpha^4 \equiv 1 + \pi^5 \pmod{\pi^6}$.  Now, suppose $$a_0\beta_0^4 + a_1\beta_1^4 = b_0 + b_1\pi + b_2\pi^2 + \ldots + b_5\pi^5 + \ldots.$$
with $a_i,\beta_i \in \mathcal{O}^\times$.
Then $$a_0\beta_0^4 + a_1(\beta_1\alpha)^4 = b_0 + b_1\pi + b_2\pi^2 + \ldots + (1+b_5)\pi^5 + \ldots.$$

If $b_5 = 1$, then $(1 + b_5) \equiv 0 \pmod{\pi}$.  Thus the coefficient of the $\pi^5$ term gets switched from 0 to 1, or vice-versa, possibly changing the values of the higher order terms, while leaving the lower order terms unchanged.  The consequence of this is that for a variable formed by the corresponding contraction, we may choose the coefficient of the $\pi^5$ term. (Here the level of the term $\pi^5$ is considered relative to the level of the original variables $x_0$ and $x_1$, not the resulting variable.)

We will say that a variable is \textit{free} at some level if the coefficient corresponding to that level can be chosen in this way.  Further, any variable resulting from a contraction or series of contractions involving a variable \textit{free at level k} will retain this very useful property.  Because of this, we will (for the sake of convenience) refer to the originating variables as free at this level also, even though they are not the result of a contraction.

As an example, consider a pair of variables in level 0 which could be contracted to a variable in level 3, and thence contracted with a variable there to one in level 5, i.e., one with a $\pi^5$-coefficient of 1 (relative to level 0).  By choosing to form the contraction in the way described so the $\pi^5$-coefficient is 0, we instead contract to a variable which has level at least 6.  (One may thus think of getting level 5 ``for free").  Similarly, if after the same series of contractions we would create a variable in a level greater than 5, we may instead arrange for the new variable to be precisely in level 5.

We now add to this notion of free variable the basic contractions which can be performed, depending only on the $\pi$- and $\pi^2$-coefficients of a pair of variables in the same level.

The results of the following two lemmas will be used frequently and without reference.

\begin{lemma}\label{contractions}~
\begin{enumerate}
    \item Two variables in the same level with differing $\pi$-coefficients can be contracted to a variable exactly one level up.
    
    \item Two variables in the same level with the same $\pi$- and $\pi^2$-coefficients can be contracted to a variable exactly two levels up.
    
    \item Two variables in the same level with the same $\pi$-coefficient and differing $\pi^2$-coefficients can be contracted to a variable at least three levels up.
    
    \item If a pair of variables in level $k$ are contracted, the resulting variable is free at level $k+5$.
\end{enumerate}
\end{lemma}
\begin{proof}

$(1 + a_1\pi + a_2\pi^2) + (1 + b_1\pi + b_2\pi^2) \equiv 2  + (a_1 + b_1)\pi + (a_2 + b_2)\pi^2  \equiv (a_1 + b_1)\pi + (1 + a_2 + b_2)\pi^2 \pmod{\pi^3}$.

\bigskip

(1) If $a_1 \ne b_1$, then $(a_1 + b_1)\pi + (1 + a_2 + b_2)\pi^2 \equiv \pi \pmod{\pi^2}$.

\bigskip

(2) If $a_1 = b_1$ and $a_2 = b_2$, then $(a_1 + b_1)\pi + (1 + a_2 + b_2)\pi^2 \equiv 2\pi + 3\pi^2 \equiv \pi^2 \pmod{\pi^3}$.

\bigskip

(3) If $a_1 = b_1$ and $a_2 \ne b_2$, then $(a_1 + b_1)\pi + (1 + a_2 + b_2)\pi^2 \equiv 2\pi + 2\pi^2 \equiv 0 \pmod{\pi^3}$.

\bigskip

Statement (4) follows from the discussion above.
\end{proof}

\begin{lemma}
\label{quartic_pigeonhole}
If there are at least three variables in the same level, a pair may be contracted at least two levels.  If there are at least five variables in the same level, a pair may be contracted exactly two levels.
\end{lemma}
\begin{proof}
By the pigeonhole principle, among three variables there are two with the same $\pi$-coefficient, and among five variables, there are two with the same $\pi,\pi^2$-coefficients.  The contractions follow from Lemma \ref{contractions}.
\end{proof}

The following is a more complicated contraction involving four variables.

\begin{lemma} \label{con4}
Suppose there are four variables in the same level $k$ having the same $\pi$-coefficient, and which can be used to form two pairs, the variables of each pair being in the same $\pi^2,\pi^3$-coefficient class.  Then those four variables can be contracted to a variable at least four levels up which is free at level $k+5$.
\end{lemma}
\begin{proof}
Let $a,b,c,d \in \mathcal{O}^\times$, with $a = 1 + a_1\pi + a_2\pi^2 + \ldots$, and the other units following the analogous convention, and with $(a,b), (c,d)$ forming the pairs of the hypothesis.  Then,
\begin{align*}
    a + &b + c + d = 
    4 + 4a_1\pi + 2(a_2 + c_2)\pi^2 + 2(a_3 + c_3)\pi^3 + (a_4 + b_4 + c_4 + d_4)\pi^4 + \ldots \\
    & = u^2\pi^4 + u^2a_1\pi^5 + u(a_2 + c_2)\pi^4 + u(a_3 + c_3)\pi^5 + (a_4 + b_4 + c_4 + d_4)\pi^4 + \ldots \\
    & \equiv 0 \pmod{\pi^4}
\end{align*}
The variable formed from this contraction is free at level $k+5$ by the same calculation used to prove Lemma \ref{contractions} (4).
\end{proof}

\begin{lemma} \label{quartic_slide}
Suppose that (\ref{quartic_form}) has two variables in level $k$, and for each of levels $k+1$, $k+2$, \ldots, $k+t-1$, either one of the two chosen variables in level $k$ is free at that level, or there is a distinct variable at that level.

Then contractions can be performed to produce a variable at level at least $k+t$.  Alternatively, contractions can be performed to produce a variable exactly at any level at which any of the initial variables was free.
\end{lemma}
\begin{proof}
Any two variables in the same level can be contracted to a variable at least one level higher.  By repeated contractions, bypassing (or stopping at) levels on which the variables are free, we obtain the desired variable.
\end{proof}

\section{Archetypal Forms}

\begin{lemma}
\label{3021}
If (\ref{quartic_form}) is of type $(3,0,2,1)$, then it has a nontrivial zero.
\end{lemma}
\begin{proof}
Contract a pair from level 0 to a variable at least two levels higher, say level $k$.  If $k \ge 7$, then a zero follows from Hensel's Lemma.  Otherwise, we will use Lemma \ref{quartic_slide} to produce a variable in level at least 7.

There is at least one variable remaining in level $0 \equiv 4 \pmod{d}$.  Treat this as a variable in level 4.  The variable resulting from the contraction is free at level 5, and so we may assume that $k \ne 5$.  The two variables in level 2 can be treated as one variable in level 2 and one in level 6.  Thus level $k$ contains two variables (including the one produced by the contraction) and levels $k+1$ through 6 either contain a variable or are free with respect to the variable formed by the initial contraction.  A zero follows from Lemma \ref{quartic_slide} and Hensel's Lemma.
\end{proof}

\begin{corollary}
\label{5011}
If (\ref{quartic_form}) is of type $(5,0,1,1)$, then it has a nontrivial zero.
\end{corollary}
\begin{proof}
By Lemma \ref{quartic_pigeonhole}, contract a pair from level 0 to level 2.  The resulting form is of type $(3,0,2,1)$ and a zero follows from Lemma \ref{3021}.
\end{proof}

\begin{corollary}
\label{7001}
If (\ref{quartic_form}) is of type $(7,0,0,1)$, then it has a nontrivial zero.
\end{corollary}
\begin{proof}
Contract a pair from level 0 to level 2.  A zero follows from Corollary \ref{5011}.
\end{proof}

\begin{lemma}
\label{6010}
If (\ref{quartic_form}) is of type $(6,0,1,0)$, then it has a nontrivial zero.
\end{lemma}
\begin{proof}
Suppose two pairs contract from level 0 to level 2.  The two resulting variables are free at level 5.  By the pigeonhole principle, two of the three variables now in level 2 have the same $\pi$-coefficient, and so a pair can be contracted to a variable in level $k \ge 4$.  Because this variable is also free at level 5, assume $k$ is either 4 or 6.  Because the remaining variables in levels 0 and 2 can be treated as variables in levels 4 and 6, respectively, level $k$ contains two variables and levels $k+1$ through 6 contain a variable or are free with respect to the variable in level $k$ resulting from the previous contractions.  A zero follows from Lemma \ref{quartic_slide} and Hensel's Lemma.

Thus assume there is at most one pair in level 0 having both variables in the same $\pi,\pi^2$-coefficient class.  It follows that there are three variables in some $\pi,\pi^2$-coefficient class, and one in each of the other three classes.  Thus there are two pairs which can be contracted to two variables at least three levels up, each free at level 5.  If any pair contracts at least four levels up, a zero follows as above from Lemma \ref{quartic_slide} (with $k=4$ or $k=6$) and Hensel's Lemma.  Thus assume both pairs contract to level 3.  A zero again follows from Lemma \ref{quartic_slide} (with $k=3$) and Hensel's Lemma.
\end{proof}

\begin{corollary}
\label{5030}
If (\ref{quartic_form}) is of type $(5,0,3,0)$, then it has a nontrivial zero.
\end{corollary}
\begin{proof}
By Corollary \ref{5011}, assume all variables in level 2 have the same $\pi$-coefficient, for otherwise a pair could be contracted to level 3, creating a form of type $(5, 0, 1, 1)$.  By the pigeonhole principle, two of these variables are in the same $\pi,\pi^2$-coefficient class.  Contract them to level $4 \equiv 0 \pmod{4}$.  A zero follows from Lemma \ref{6010}.
\end{proof}

\begin{corollary}
\label{8000}
If (\ref{quartic_form}) is of type $(8,0,0,0)$, then it has a nontrivial zero.
\end{corollary}
\begin{proof}
Contract a pair from level 0 to level 2.  A zero follows from Lemma \ref{6010}.
\end{proof}


\section{Remaining Cases}

\begin{lemma}
\label{quartic_max7}
Suppose that $s \ge 11$ and after normalization (\ref{quartic_form}) has at least seven variables in level 0.  Then (\ref{quartic_form}) has a nontrivial zero.  
\end{lemma}
\begin{proof}
By Corollary \ref{8000}, assume level 0 contains exactly seven variables.  By Corollary \ref{7001}, assume level 3 is empty, and by Lemma \ref{6010}, assume level 2 is empty.  Then level 1 has at least four variables.  By Lemma \ref{6010}, assume they have the same $\pi$-coefficient.  By the pigeonhole principle, there is a pair in level 1 which can be contracted to level 3.  A zero follows from Corollary \ref{7001}.
\end{proof}

\begin{lemma}
\label{quartic_max6}
Suppose that $s \ge 11$ and after normalization (\ref{quartic_form}) has at least six variables in level 0.  Then (\ref{quartic_form}) has a nontrivial zero.  
\end{lemma}
\begin{proof}
By Lemma \ref{quartic_max7}, assume level 0 contains exactly six variables.  By Lemma \ref{6010}, assume level 2 is empty.  By normalization, level 1 contains at least three variables, and so by Lemma \ref{3021}, assume level 3 contains at most one variable.  Thus assume (\ref{quartic_form}) is of type $(6,4,0,1)$ or $(6,5,0,0)$.  By Lemma \ref{6010} assume the variables in level 1 all have the same $\pi$-coefficient.

Suppose first that (\ref{quartic_form}) is of type $(6,5,0,0)$.  Note that any variable formed from a contraction from level 1 will be free at level 6.  Supposing a pair in level 1 contracts to level 4, by the pigeonhole principle there is a pair among the three remaining variables which contracts to level 3, and a zero follows from Corollary \ref{7001}.  If a pair in level 1 contracts to at least level 5, then the resulting variable is free at level 6, and another variable from level 1 may be treated as being in level 5.  Thus by Lemma \ref{quartic_slide}, a variable can be formed at level 6, and a zero follows from Lemma \ref{6010}.  Thus assume every pair of variables in level 1 can only contract to level 3, and so they all belong to the same $\pi,\pi^2$-class.  By the pigeonhole principle and Lemma \ref{con4}, four of these variables contract to a variable in at least level 5.  Because this variable is free at level 6, by Lemma \ref{quartic_slide}, a variable can be formed at level 6 as before, and a zero follows from Lemma \ref{6010}.

Suppose now that (\ref{quartic_form}) is of type $(6,4,0,1)$.  By Corollary \ref{5011}, assume the variables in level 0 have the same $\pi$-coefficient.  If a pair from level 0 contracts to level 3, a zero follows from Lemma \ref{3021}.  If a pair from level 0 contracts to level 4 or higher, by Lemma \ref{quartic_slide} a variable can be formed in level 5, and a zero follows from Corollary \ref{5011}.  Thus assume all variables in level 0 are in the same $\pi,\pi^2$-class.  By the pigeonhole principle and Lemma \ref{con4}, four variables contract to a variable in level at least 4.  By Lemma \ref{quartic_slide}, a variable can be formed in level 5, and a zero follows from Corollary \ref{5011}.
\end{proof}

\begin{lemma}
\label{quartic_max5}
Suppose that $s \ge 11$ and after normalization (\ref{quartic_form}) has at least five variables in level 0.  Then (\ref{quartic_form}) has a nontrivial zero.  
\end{lemma}
\begin{proof}
By Lemma \ref{quartic_max6}, assume level 0 contains exactly five variables.  By Corollary \ref{5030}, assume level 2 contains at most two variables, and so by normalization level 1 contains at least two variables.

If level 2 contains one or two variables, then by Corollary \ref{5011}, assume level 3 is empty, and so level 1 contains five or four variables, respectively.  If level 1 contains five variables, by the pigeonhole principle there is a pair that can be contracted to level 3, and a zero follows from Corollary \ref{5011}.  If, on the other hand, level 2 contains two variables, contract a pair from level 0 to level 2, and a zero follows from Lemma \ref{3021}.

Thus, assume level 2 is empty.  By normalization, level 1 contains at least four variables.  By Lemma \ref{3021}, assume level 3 contains at most one variable, and by normalization level 1 contains at least five variables.  If it contains exactly five, then level 3 contains one, and a zero follows from Corollary \ref{5011}.

Thus, assume (\ref{quartic_form}) is of type $(5,6,0,0)$.  We first describe two possibilities for forming variables using those in level 0.  By Corollary \ref{7001}, assume all variables in level 0 have the same $\pi$-coefficient.  By Lemma \ref{6010}, assume no pair from level 0 contracts to level 3.  If a pair from level 0 contracts at least to level 4, then by Lemma \ref{quartic_slide} it can be used to form a variable in level 5, and a zero follows from Corollary \ref{7001}.  Thus assume all variables in level 0 are in the same $\pi,\pi^2$-class and so by the pigeonhole principle there are two pairs in the same $\pi^2,\pi^3$-class.  Therefore by Lemma \ref{con4} four of these variables can be contracted to a variable in at least level 4.  If they contract to at least level 5, they can be contracted to exactly level 5 and a zero follows from Corollary \ref{7001}.  Thus, assume a variable $y$ which is free at level 5 could be formed at level 4 by consuming four variables in level 0.  Further by Lemma \ref{quartic_slide} the variable $y$ and the remaining variable from level 0 could be consumed to form a variable $y'$ at level 5.

Now, returning to the original form of type $(5,6,0,0)$, we consider contractions which can be performed with variables in level 1.  A pair can be contracted from level 1 to level 3, and so by Corollary \ref{5011}, assume all variables in level 1 have the same $\pi$-coefficient.  If a pair from level 1 contracts to at least level 6, it can be contracted exactly to level 6 (and regarded as a variable in level 2), another pair can be contracted to level 3, and a zero follows from Corollary \ref{5011}.  If a pair from level 1 contracts to level 5, then the resulting variable is free at level 6, and thus can be contracted with $y'$ to a variable at level at least 7, and a zero follows from Hensel's Lemma.  If a pair from level 1 contracts to level 4, then the resulting variable is free at level 6, may be contracted with $y$ which is free at level 5, resulting in a variable at level at least 7, and a zero follows from Hensel's Lemma.  Thus assume all pairs in level 1 contract to level 3, and so are all in the same $\pi,\pi^2$-class.  Therefore, by the pigeonhole principle and Lemma \ref{con4} four variables can be contracted to a variable at level at least 5. If they contract to level at least 6, they can be contracted exactly to level 6, the remaining pair in level 1 can be contracted to level 3, and a zero follows from Corollary \ref{5011}.  Thus, the four variables contract to a variable in level 5 which can then be contracted with $y'$, resulting in a variable at level at least 7, and a zero follows from Hensel's Lemma.
\end{proof}

\begin{lemma}
\label{quartic_max4}
Suppose that $s \ge 11$ and after normalization (\ref{quartic_form}) has at least four variables in level 0.  Then (\ref{quartic_form}) has a nontrivial zero.  
\end{lemma}
\begin{proof}
By Lemma \ref{quartic_max5}, assume level 0 has exactly four variables.  By normalization, level 1 has at least two variables.  By Lemma \ref{3021}, assume level 2 has at most two variables, and so by normalization level 1 has at least three variables.  By Lemma \ref{3021}, assume level 3 has at most one variable, and further by the same lemma that either level 2 has at most one variable or level 3 is empty.  In either case, levels 2 and 3 together have at most two variables, and so level 1 has at least five variables.  By Corollary \ref{5011}, assume level 3 is empty, and by Corollary \ref{7001} that level 1 has at most six variables.  Thus assume (\ref{quartic_form}) is of type $(4,6,1,0)$ or $(4,5,2,0)$.  If it is of type $(4,5,2,0)$, then a pair from level 1 can be contracted to level 3, and a zero follows from Lemma \ref{3021}.

Thus, assume (\ref{quartic_form}) is of type $(4,6,1,0)$.  By Corollary \ref{7001}, assume all variables in level 0 have the same $\pi$-coefficient.  If a pair from level 0 contracts to level 3, a zero follows from Lemma \ref{6010}, and if a pair contracts to level at least 4, then a zero follows from Lemma \ref{quartic_slide} (with $k=4$) and Hensel's Lemma.  Thus assume that two pairs in level 0 contract to level 2.  If any two of the three resulting variables in level 2 have different $\pi$-coefficients, then a zero follows from Lemma \ref{6010}.  Thus, contract just one pair from level 0 to level 2, and the resulting pair contracts at least to level 4.  We may assume that it contracts to level 4 or 5 (because it is free at level 5).  By Lemma \ref{quartic_slide}, if the variable was formed at level 4, it can be contracted to create a variable in level 5, which by a change of variables can be moved to level 1.  A zero follows from Corollary \ref{7001}.
\end{proof}

\begin{lemma}
\label{quartic_max3}
Suppose that $s \ge 11$ and after normalization (\ref{quartic_form}) has at least three variables in level 0.  Then (\ref{quartic_form}) has a nontrivial zero.  
\end{lemma}
\begin{proof}
By Lemma \ref{quartic_max4}, assume level 0 contains exactly three variables.  Suppose level 3 contains one variable.  Then by Lemma \ref{3021}, assume level 2 contains at most one variable, and so level 1 contains at least six variables, and a zero follows from Lemma \ref{6010}.

Thus assume level 3 is empty.  By Corollary \ref{5030}, assume level 2 has at most four variables, and so level 1 has at least four variables.  By Lemma \ref{3021}, assume level 2 has at most two variables, and so level 1 has at least six variables.  By Corollary \ref{7001}, assume level 1 has exactly six variables, and thus level 2 has exactly two.  Contract a pair from level 1 to level 3.  A zero follows from Lemma \ref{3021}.
\end{proof}

Because $s \ge 11$, level 0 will contain at least three variables after normalization, and so by Lemma \ref{quartic_max3} we have that $\Gamma^*(4,K) \le 11$.  In order to establish equality we will demonstrate an anisotropic form in ten variables.

\begin{theorem}
Let $G = (x_0^4 + x_1^4 + x_2^4 + x_3^4)$ and $H = (x_4^4 + x_5^4 + x_6^4 + x_7^4 + x_8^4 + x_9^4)$. 

\smallskip

For $K \in \{\mathbb{Q}_2(\sqrt{2}), \mathbb{Q}_2(\sqrt{10})\}$, $F = G + \pi H$ has no nontrivial zero.

\smallskip

For $K \in \{\mathbb{Q}_2(\sqrt{-2}), \mathbb{Q}_2(\sqrt{-10})\}$, $F = G + (\pi + \pi^2)H$ has no nontrivial zero.   
\end{theorem}
\begin{proof}
Let $v_\pi$ be the valuation function giving the greatest power of $\pi$ dividing $\alpha \in \mathcal{O}$.  Because the only fourth powers modulo $\pi^5$ are 0 and 1, $G$ has no primitive zero modulo $\pi^5$, and therefore no nontrivial zero in $K$.  Thus, for $F$ to have a nontrivial zero $a = (a_0, a_1, \ldots, a_9)$, we must have that $v_\pi(G(a)) = v_\pi(H(a)) + 1 \ne \infty$ for some substitution of the variables, and so one of $v_\pi(G(a))$ and $v_\pi(H(a))$ is odd and the other is even.  Without loss of generality, we may assume that all variables take values in $\mathcal{O}$.  To determine the valuations of values that can be represented by $G$, we may reduce to the case where at least one of its variables takes a unit value.  Recall from Lemma \ref{powers} that all unit fourth powers (modulo $\pi^7$) are congruent to either $1$ or $1+\pi^5$ if $K \in \{\mathbb{Q}_2(\sqrt{2}),\mathbb{Q}_2(\sqrt{10})\}$, and $1$ or $1 + \pi^5 + \pi^6$ if $K \in \{\mathbb{Q}_2(\sqrt{-2}),\mathbb{Q}_2(\sqrt{-10})\}$.  If an odd number of variables in $G$ take a unit value, then $v_\pi(G(a)) = 0$; if two variables take a unit value, then $v_\pi(G(a)) = 2$; and if four variables take a unit value, then $v_\pi(G(a)) = 4$.  It follows that $G$ represents no values with odd valuation and so $v_\pi(H(a))$ is odd.

Again, recall that $2 \equiv \pi^2 \pmod{\pi^6}$ for $K \in \{\mathbb{Q}_2(\sqrt{2}),\mathbb{Q}_2(\sqrt{10})\}$, and $2 \equiv \pi^2 + \pi^4 \pmod{\pi^6}$ for $K \in \{\mathbb{Q}_2(\sqrt{-2}),\mathbb{Q}_2(\sqrt{-10})\}$, and so in either case $4 \equiv \pi^4 \pmod{\pi^7}$.  We again assume at least one of the variables in $H$ takes a unit value.  By the above reasoning, an even number of variables of $H$ must take unit values.  If two or six variables take unit values, then $v_\pi(H(a)) = 2$.  Thus, assume exactly four variables take unit values.  The values representable by these four variables (modulo $\pi^7$) are $\pi^4$, $\pi^4 + \pi^5$ for $K \in \{\mathbb{Q}_2(\sqrt{2}),\mathbb{Q}_2(\sqrt{10})\}$, and $\pi^4$, $\pi^4 + \pi^5 + \pi^6$ for $K \in \{\mathbb{Q}_2(\sqrt{-2}),\mathbb{Q}_2(\sqrt{-10})\}$.  If neither or both of the remaining two variables are divisible by $\pi^2$, then $v_\pi(H(a)) = 4$.  Thus assume both are divisible by $\pi$ with exactly one divisible by $\pi^2$.  It follows that the only values (modulo $\pi^{7+4j}$) with odd valuation represented by $H$ are $\pi^{4j}(2\pi^4 + \pi^5) \equiv \pi^{4j}(\pi^5 + \pi^6)$ if $K \in \{\mathbb{Q}_2(\sqrt{2}), \mathbb{Q}_2(\sqrt{10})\}$, and $\pi^{4j}(2\pi^4 + \pi^5 + \pi^6) \equiv \pi^{4j}(\pi^5)$ if $K \in \{\mathbb{Q}_2(\sqrt{-2}), \mathbb{Q}_2(\sqrt{-10})\}$, where $\pi^j$ divides all of the values taken by the variables of $H$.

$G$ can only represent a value with valuation greater than 4 if all of its variables take nonunit values.  Without loss of generality, assume one of the variables of $F$ takes a unit value, and so $v_\pi(H(a)) = 5$, i.e., $j=0$.  Further, $G$ can only represent a value with valuation 6 if the values taken by two of its variables have valuation 1, and the values taken by the other two have valuation at least 2.  But then $G(a) \equiv \pi^6$ modulo $\pi^8$, and so we have for $K \in \{ Q_2(\sqrt{2}), Q_2(\sqrt{10}) \}$ that $\pi^6 + \pi (\pi^5 + \pi^6 + b \pi^7) \equiv \pi^7 \bmod \pi^8$, and for $K \in \{ Q_2(\sqrt{-2}), Q_2(\sqrt{-10}) \}$ that $\pi^6 + (\pi + \pi^2) (\pi^5 + b \pi^7) \equiv \pi^7 \bmod \pi^8$.  Thus $F$ has no primitive zero modulo $\pi^8$, and thus no nontrivial zero in $K$.
\end{proof}

This completes the proof.

\bibliographystyle{plain}
\bibliography{biblio}

\end{document}